\documentclass[final,1p,times]{elsarticle}

\journal{Advances in Mathematics}

\usepackage[pdfstartview=FitH,
            colorlinks,
           linkcolor=reference,
            citecolor=citation,
            urlcolor=e-mail,
           backref]{hyperref}

\usepackage{amsmath,amsthm,amscd,leftidx,bbm}
\usepackage{tikz}

\definecolor{todo}{rgb}{1,0,0}
\definecolor{answer}{rgb}{0,0,1}
\definecolor{new}{rgb}{1,0,1}
\definecolor{conditional}{rgb}{0,1,0}
\definecolor{e-mail}{rgb}{0,.40,.80}
\definecolor{reference}{rgb}{.20,.60,.22}
\definecolor{mrnumber}{rgb}{.80,.40,0}
\definecolor{citation}{rgb}{0,.40,.80}

\DeclareMathAlphabet{\mathpzc}{OT1}{pzc}{m}{it}
\RequirePackage{dsfont}

\newcommand \one {{\mathbbm 1}}
\DeclareMathOperator{\At}{\mathrm F}
\DeclareMathOperator{\Ker}{\mathrm ker}
\DeclareMathOperator{\Img}{\mathrm Im}
\DeclareMathOperator{\Id}{\mathrm id}
\DeclareMathOperator{\End}{\mathrm End}

\def\QX{{\mathbb Q}}

\def\ZX{{\mathbb Z}}

\def\GL{{\bf GL}}

\def\ST{{\bf ST}}
\def\SL{{\bf SL}}

\def\sl{{\rm sl}}

\def\C{{\rm C}}

\def\Hom{{\rm Hom}}

\def\Ga{{\bf G_a}}
\def\Gm{{\bf G_m}}

\DeclareMathOperator{\K}{\bf k}
\DeclareMathOperator{\F}{\mathcal F}
\DeclareMathOperator{\Z}{\mathbb Z}
\DeclareMathOperator{\Aut}{\bf Aut}
\DeclareMathOperator{\Q}{\mathbb Q}
\DeclareMathOperator{\I}{\bf I}
\DeclareMathOperator{\Span}{span}

\DeclareMathOperator{\Ru}{R_u}
\DeclareMathOperator{\Rep}{\bf Rep}
\DeclareMathOperator{\id}{id}
\DeclareMathOperator{\U}{\mathcal U}
\def\Ob{{\mathrm Ob}}
\DeclareMathOperator{\Vect}{\bf Vect}
\DeclareMathOperator{\Cat}{\mathcal{C}}

\DeclareMathOperator{\Char}{char}

\theoremstyle{plain}
\newtheorem{theorem}{Theorem}[section]
\newtheorem{lemma}[theorem]{Lemma}
\newtheorem{proposition}[theorem]{Proposition}

\newtheorem{corollary}[theorem]{Corollary}

\theoremstyle{definition}
\newtheorem{definition}[theorem]{Definition}

\newtheorem{example}[theorem]{Example}

\theoremstyle{remark}
\newtheorem{remark}[theorem]{Remark}

\DeclareFontEncoding{OT2}{}{}

\newcommand{\Le}{\leqslant}
\newcommand{\Ge}{\geqslant}

 \usepackage{mathptmx}

\begin{document}
\begin{frontmatter}

\title{Zariski closures of reductive linear differential algebraic groups}
\author{Andrey Minchenko\fnref{thank1,formeraddress}}
\ead{an.minchenko@gmail.com}
\address{University of Western Ontario, Department of Mathematics,
 London, ON N6A 5B7, Canada}

\author{Alexey Ovchinnikov\fnref{thank2}}
\ead{aiovchin@gmail.com}
\address{City University of New York,
Queens College, Department of Mathematics,
65-30 Kissena Blvd, Flushing, NY 11367, USA
}

\fntext[thank1]{This author was supported by  the NSF Grant DMS-0901570.}
\fntext[formeraddress]{The former address is Cornell University, Department of Mathematics,
 Ithaca, NY 14853-4201, USA}
\fntext[thank2]{This author was supported by the grants: NSF  CCF-0901175, 0964875, and 0952591 and  PSC-CUNY~60001-40~41.}
\date\today


\begin{abstract}  Linear differential algebraic groups (LDAGs) appear as Galois groups of systems of linear differential and difference equations with parameters. These groups measure differential-algebraic dependencies among solutions of the equations. LDAGs are now also used in factoring partial differential operators.
In this paper, we study Zariski closures of LDAGs. In particular, we give a Tannakian characterization of algebraic groups that
are Zariski closures of a given LDAG. Moreover, we show that the Zariski closures that correspond to representations of minimal dimension of  a reductive LDAG are all isomorphic. 
In addition, we give a Tannakian description of simple LDAGs.
This substantially extends the classical results of P.~Cassidy and, we hope, will have an impact on developing algorithms that compute differential Galois groups of the above equations  and factoring partial differential operators.
\end{abstract}
\end{frontmatter}

\section{Introduction}
In this paper, we continue developing  a Tannakian approach to the representation theory
of linear differential algebraic groups (LDAGs) over differential fields with several commuting derivations started in~\cite{OvchRecoverGroup,OvchTannakian}. We combine it with the classical
approach to these groups initiated by
Cassidy in \cite{Cassidy,CassidyRep}. A noncommutative LDAG  is called simple if it has no connected normal differential
algebraic subgroups \cite{CassidyClassification,Buium}. A LDAG is called reductive if  its differential unipotent radical is trivial~\cite{CassidyUnipotent,Pillay3}.

Our main result is a characterization of
reductive LDAGs in terms of Zariski closures using Cassidy's results~\cite{CassidyClassification} on simple LDAGs. We show that a LDAG $G$ is reductive if and only if its Zariski closure
in a faithful representation of $G$ of minimal dimension is a reductive linear algebraic
group (see Theorem~\ref{main}). This gives a complete Tannakian description of the category of differential representations of a reductive LDAG (see Corollary~\ref{cor:main}).

In general, the Zariski closure of a LDAG $G$ does depend on the embedding of $G$ into $\GL(V)$.
Moreover, if $G$ is reductive, its Zariski closure does not have to be so (see Examples~\ref{ex:nonfaithful} and~\ref{sl2example}). As an application of our results for reductive LDAGs we give a Tannakian characterization of simple LDAGs (see Theorems~\ref{thm:algsimple} and~\ref{thm:diffsimple}).  
This should have applications to factoring partial differential operators in the sense of~\cite{MichaelPhyllisFactoring}.

In order to show the uniqueness result, we use the Tannakian approach and, in particular, show that
if one takes a generator $X$ of a neutral {\it differential} Tannakian category  \cite{OvchTannakian,difftann,Moshe}, then the neutral Tannakian category  \cite{Saavedra,Deligne} generated by $X$ is the category of representations of the Zariski closure of the LDAG that
corresponds to $X$ (see Theorem~\ref{ZariskiClosure}). 
In general, the category of representations of a reductive LDAG is not semisimple. This is one of the main difficulties of the theory. Using the results of~\cite{OvchRecoverGroup}, we show that it is semisimple if and only if the group is not only reductive, but is also
conjugate to a group of matrices with entries that are constant under our derivations (see Theorem~\ref{semisimplethm}).

LDAGs appear as Galois groups of systems of linear differential and difference equations with parameters~\cite{PhyllisMichael,CharlotteMichael,CharlotteComp,CharlotteLucia}. These groups measure differential-algebraic dependencies among solutions of the equations. At present, we do not have an algorithm that computes this Galois group. However,  solving such an important problem
becomes more feasible with the results we present in this paper. Indeed, a usual algorithm that computes the Galois group of a linear differential equation (without parameters) generally operates with a list of groups that can possibly occur and step-by-step eliminates the choices (see, for example, \cite{Kovacic1,FelixMichael1,FelixMichael2,FelixMichael3,HRUW,FelixJAW,Michael}). 

Hence, by eliminating non-simple and non-reductive LDAGs  our result will contribute to such a step in a future algorithm that computes the Galois group of a linear differential and difference equation with parameters. For second order differential equations, this might combined with the results on differential algebraic subgroups and differential representations of $\SL_2$ \cite{SitSL2,MinOvRepSL2}. Also, algebraically finite dimensional LDAGs have been studied in \cite{Buium2,Buium1}. Their further connections with the generalized differential Galois theory appeared in \cite{Pillay2,Pillay1}.

The rest of the paper is organized as follows.  Section~\ref{Basics}
gives formal definitions and properties  of linear differential
algebraic groups including  prolongations of representations.  In Sections~\ref{unipotentsection} and~\ref{reductivesection}, we introduce unipotent LDAGs (Definition~\ref{unipotentdef} and Lemma~\ref{unipotentlemma}), show existence of differential unipotent radicals,
and characterize semisimple categories of representation of LDAGs. 
 In Section~\ref{uniqueness}, we show our main result: the uniqueness of
Zariski closures in faithful representations of minimal dimension for reductive LDAGs and the application to describing all simple LDAGs in Tannakian terms. Section~\ref{diffChevalley} contains the differential Chevalley theorem on realizing any LDAG as a stabilizer of a line defined over the field of definition of the group. This result is not only new and interesting on its own as an alternative way of viewing the Tannakian approach but is also used directly in~\cite{CharlotteLucia}.

\section{Basic definitions}\label{Basics}
 A $\Delta$-ring $R$, where $\Delta =\{\partial_1,\ldots,\partial_m\}$, is a commutative associative ring with the unit  and commuting derivations $\partial_i: R\to R$ such that
$$
\partial_i(a+b) = \partial_i(a)+\partial_i(b),\quad \partial_i(ab) =
\partial_i(a)b + a\partial_i(b)
$$
for all $a, b \in R$ and $i$, $1\Le i \Le m$. If $\K$ is a field and
a $\Delta$-ring then $\K$ is called a $\Delta$-field. We restrict ourselves to the case of
$$
\Char\K = 0.
$$
For example, $\Q$ is a $\Delta$-field with the unique
possible derivation (which is the zero one). The field
$\mathbb{C}(t)$ is also a $\Delta$-field with $\partial(t) = f,$
and this $f$ can be any element of $\mathbb{C}(t).$
Let
$$
\Theta = \left\{\partial_1^{i_1}\cdot\ldots\cdot\partial_m^{i_m}\:|\: i_j\in \Z_{\Ge 0}\right\}.
$$
The action of $\Delta$  on $\Delta$-ring $R$ induces an action of $\Theta$ on $R$.

Let $R$ be a $\Delta$-ring. If $B$ is an $R$-algebra, then $B$ is a $\Delta$-$R$-algebra
if the action of $\Delta$ on $B$ extends the
action of $\Delta$ on $R$.
Let $Y = \{y_1,\ldots,y_n\}$ be a set of variables. We differentiate them:
$$
\Theta Y := \left\{\theta y_j
\:\big|\: \theta \in \Theta,\ 1\Le j\Le n\right\}.
$$
The ring of differential polynomials $R\{Y\}$ in
the differential indeterminates $Y$
over a $\Delta$-ring $R$ is
the ring of commutative polynomials $R[\Theta Y]$
in infinitely many algebraically independent variables $\Theta Y$ with
derivations $\Delta$ that 
extend the $\Delta$-action on $R$ as follows:
$$
\partial_i\left(\theta y_j\right) := (\partial_i\theta)y_j
$$
for all $1\Le i\Le m$ and $1 \Le j \Le n$.
An ideal $I$ in a $\partial$-ring $R$ is called differential if it is stable under the action of
$\partial$, that is,
$$
\partial_i(a) \in I \ \ \text{for all} \ a \in I,\ 1\Le i \Le m.
$$

We shall recall some definitions and results from differential
algebra (see for more detailed information \cite{Cassidy,Kol})
leading up to the ``classical definition'' of a linear differential
algebraic group.
Let $\U$ be a  differentially closed field containing $\K$ (see \cite[Definition 3.2]{PhyllisMichael} and the references given there).
Let also $C\subset\U$ be its subfield of constants\footnote{One can show that the field $C$ is algebraically closed.}, that is, $$C = \bigcap_{1\Le i\Le m}\ker\partial_i.$$

\begin{definition} For a differential field extension $K\supset \K$, a {\it Kolchin closed} subset $W(K)$ of $K^n$ over $\K$ is the set of common zeroes
of a system of differential algebraic equations with coefficients in $\K,$ that is, for $f_1,\ldots,f_k \in \K\{Y\}$ we define
$$
W(K) = \left\{ a \in K^n\:|\: f_1(a)=\ldots=f_k(a) = 0\right\}.
$$
\end{definition}

There is a bijective correspondence between Kolchin closed subsets
$W$ of $\U^n$ defined over $\K$ and radical differential ideals
$\I(W) \subset \K\{y_1,\ldots,y_n\}$
generated by the differential polynomials $f_1,\ldots,f_k$ that define $W$.
In fact, the $\partial$-ring $\K\{Y\}$ is
Ritt-Noetherian, meaning that every radical
differential ideal is the radical of a finitely
generated differential ideal, by the Ritt-Raudenbush basis theorem.
Given a Kolchin closed subset $W$ of
$\U^n$ defined over $\K$, we let the coordinate ring $\K\{W\}$
be:
$$
\K\{W\} = k\{y_1,\ldots,y_n\}\big/\I(W).
$$
A differential polynomial map $\varphi : W_1\to W_2$ between Kolchin closed subsets of $\U^n,$ defined over $\K,$ is given in coordinates by differential polynomials in
$\K\{W_1\}$. Moreover, to give $$\varphi : W_1 \to W_2$$
is equivalent to defining $$\varphi^* : \K\{W_2\} \to \K\{W_1\}.$$

\begin{definition}\cite[Chapter II, Section 1, page 905]{Cassidy} A {\it linear differential
algebraic group} 
is a Kolchin closed subgroup $G$ of $\GL_n(\U),$
that is, an intersection
of a Kolchin closed subset of $\U^{n^2}$ with $\GL_n(\U)$, which is closed under
the group operations.
\end{definition}

As mentioned in the introduction, we will also use the abbreviation LAG (respectively, LDAG) for linear algebraic group (respectively, linear differential algebraic group). Note that it follows from \cite[Theorem~4.3]{Waterhouse} that the Zariski closure of a LDAG $G \subset \GL_n(\U)$ is a linear algebraic group.
Here, we identify $\GL_n(\U)$ with a Zariski closed
subset of $\U^{n^2+1}$ given by
$$\left\{(A,a)\:\big|\: (\det(A))\cdot a-1=0\right\}.$$
If $X$ is an invertible $n\times n$ matrix, we
can identify it with the pair $(X,1/\det(X))$. Hence, we may represent the coordinate ring of $\GL_n(\U)$ as
$$
\K\{X,1/\det(X)\}.$$
Denote $\GL_1$ simply by $\Gm$. Its coordinate ring
is $\K\{y,1/y\},$ where $y$ is a differential indeterminate.
The LDAG with coordinate ring $\K\{y\}$ and the usual group structure is denoted by $\Ga$.

For a group $$G\subset\GL_n(\U),$$ we denote   its the Zariski closure  in $\GL_n(\U)$ by $\overline{G}$. Then, $\overline{G}$ is a linear algebraic group over $\U$.

\begin{definition}\cite{CassidyRep} Let $G$ be a LDAG. A differential polynomial
group homomorphism  $$\rho : G \to \GL(V)$$ is called a
{\it differential representation} of $G$, where $V$ is a
finite dimensional vector space over $\K$. Such a space is
called a {\it $G$-module}. As usual, morphisms between $G$-modules are $\K$-linear maps that are $G$-equivariant. The category of differential representations of $G$ is denoted by $\Rep_G$.
\end{definition}

By \cite[Proposition~7]{Cassidy}, $\rho(G)\subset\GL(V)$ is a differential algebraic subgroup. Moreover, by \cite[Proposition~8]{Cassidy}, if $\rho$ is faithful and $G\subset\GL(W)$, then there exists a representation $$\rho^{-1} : \rho(G)\to\GL(W)$$ such that $\rho\circ\rho^{-1}=\rho^{-1}\circ\rho=\id$. 

Given a representation $\rho$ of a LDAG $G$, one can define its prolongations 
$$F_i(\rho) : G \to \GL(F_iV)
$$ with respect to $\partial_i$, $1\Le i\Le m$, as follows~\cite[Definition 4 and Theorem 1]{OvchRecoverGroup}: let $$F_i(V) =\leftidx{_{\K}}{\left((\K\oplus \K\partial_i)_{\K}\otimes_{\K} V\right)}$$
as vector spaces. Here, $\K\oplus \K\partial_i$ is considered as the right $\K$-module:
$$
\partial_i\cdot a = \partial_i(a) + a\partial_i
$$
for all $a \in \K$.
Then the action of $G$ is given by $F_i(\rho)$ as follows:
$$
F_i(\rho)(g) (1\otimes v) := 1\otimes \rho(g)(v),\quad F_i(\rho)(g)(\partial_i\otimes v) := \partial_i\otimes\rho(g)(v)
$$
for all $g \in G$ and $v \in V$. In the language of matrices, if $A_g \in \GL(V)$ corresponds to the action of $g \in G$  on $V$, then the matrix
$$
\begin{pmatrix}
A_g&\partial_i A_g\\
0&A_g
\end{pmatrix}
$$
corresponds to the action of $g$ on $F_i(V)$. The above induces an exact sequence of differential representations of $G$ (see \cite[Definition 1]{OvchTannakian} and \cite[Section 5]{Moshe}):
\begin{equation}\label{eq:prolongation}
\begin{CD}
0@>>> V @>\iota_{V,i} >> F_i(V)@>\pi_{V,i} >> V@>>> 0
\end{CD}
\end{equation}
with
$$
\iota_{V,i}( v) = 1\otimes v,\quad \pi_{V,i}(1\otimes v) = 0,\ \pi_{V,i}(\partial_i\otimes v) = v,\quad v\in V.
$$
Moreover, for $1\Le i,j\Le m$ we have $G$-isomorphisms
\begin{equation}\label{eq:commutingreps}
S_{i,j} : F_i(F_j(V)) \stackrel{\sim}\longrightarrow F_j(F_i(V))
\end{equation}
given by
\begin{align*}
&1\otimes 1\otimes v\mapsto 1\otimes1\otimes v,\quad \partial_i\otimes1\otimes v\mapsto1\otimes\partial_i\otimes v,\\
&\partial_i\otimes \partial_j\otimes v\mapsto \partial_j\otimes\partial_i\otimes v,\ 1\otimes\partial_j\otimes v\mapsto\partial_j\otimes1\otimes v,
\end{align*}
where $v \in V$.
Further prolongations $$F_i^p(\rho) : G \to \GL\left(F_i^p(V)\right)$$ are given by iterating the construction, and
further results on this approach are contained in~\cite{OvchRecoverGroup,GGO}. Moreover, \cite{GGO} develops differential Tannakian categories in a more general setting: when $\partial_1,\ldots,\partial_m$ do not necessarily commute, by rather operating with finite dimentional Lie algebras of derivations without choosing bases.

\begin{remark}
There is also a Grothendieck-style approach to linear differential algebraic groups
and their representations via representative functors, differential Hopf algebras, their
comodules, and prolongation functors. This approach does not use differentially closed fields \cite{OvchRecoverGroup,OvchTannakian}. For convenience, we will use {\bf both} languages interchangeably in what follows.
\end{remark}

 \section{Tannakian results and definitions}
 As we noted in the introduction, our plan is first to give a Tannakian characterization of
 Zariski closures of LDAGs and then use it for our main result.

 \subsection{Zariski closures of LDAGs and differential Tannakian categories}
 The following definition is taken from \cite[Definition 2]{OvchTannakian} (where they were first introduced) with a slight modification and  generalization to the case of several commuting derivations reflected in isomorphisms~\eqref{eq:commutingderivations} (see also~\cite{GGO}).

\begin{definition}\label{defin-diffTan}
A {\it neutral differential Tannakian category} over a differential field $(\K,\Delta=\{\partial_1,\ldots,\partial_m\})$ is a rigid tensor abelian category $\Cat$ with $\End_{\Cat}(\one)=\K$, where $\one$ is the unit object in $\Cat$, supplied with the following data:
\begin{itemize}
\item
functors $\At_{\Cat,i}:\Cat\to\Cat$, $1\Le i\Le m$, with exact sequences of functors
$$
\begin{CD}
0@>>>\Id_{\Cat}@>\iota_i>>\At_{\Cat,i}@>\pi_i>>\Id_{\Cat}@>>> 0,
\end{CD}
$$
that is, a functorial exact sequence
\begin{equation}\label{eq-exactAtcat}
\begin{CD}
0@>>>X @>{\iota_{X,i}}>> \At_{\Cat,i}(X)@>{\pi_{X,i}}>> X@>>> 0
\end{CD}
\end{equation}
for any object $X$ in $\Cat$, such that the differential structure induced by all the $\At_{\Cat,i}$'s on $\K$ coincides with the given $\Delta$ (see~\cite[Lemma~7]{difftann} and~\cite[Section~5.2.2]{Moshe})
\item
functorial isomorphisms
\begin{equation}\label{eq:commutingderivations}
S_{i,j} : \At_{\Cat,i}(\At_{\Cat,j}(X))\stackrel{\sim}\longrightarrow \At_{\Cat,j}(\At_{\Cat,i}(X))
\end{equation}
for all $1\Le i,j\Le m$ and $X \in \Ob(\Cat)$;
\item
functorial morphisms
\begin{equation}\label{eq-TXYcat}
T_{X,Y,i}:\At_{\Cat,i}(X\otimes Y)\to \At_{\Cat,i}(X)\otimes\At_{\Cat,i}(Y)
\end{equation}
for all objects $X$ and $Y$ in $\Cat$ and $1\Le i\Le m$;
\item
an exact tensor $k$-linear functor $\omega:\Cat\to \Vect_{\K}$ with 
 functorial isomorphisms of $\K$-vector spaces
$$
\Gamma_{X,i}:\omega(\At_{\Cat,i}(X))\stackrel{\sim}\longrightarrow F_i(\omega(X))
$$
for any object $X$ in $\Cat$, such that $\Gamma$ sends exact sequences of type \eqref{eq-exactAtcat} to exact sequences of type \eqref{eq:prolongation}, isomorphisms of type~\eqref{eq:commutingderivations} to isomorphisms of type~\eqref{eq:commutingreps}, and morphisms of type \eqref{eq-TXYcat} to the morphisms of the following type:
\begin{equation*}
T_{U,V,i}:F_i(U\otimes_{\K} V)\to F_i(U)\otimes_{\K}F_i(V)
\end{equation*}
given by the formula
\begin{align*}
1\otimes(u\otimes v)&\mapsto (1\otimes u)\otimes(1\otimes v),\\
\partial_i\otimes(u\otimes v)&\mapsto (\partial_i\otimes u)\otimes (1\otimes v)+(1\otimes u)\otimes (\partial_i\otimes v),
\end{align*}
for all $u\in U$, $v\in V$, and $1\Le i\Le m$.
Such a functor $\omega$ is called a differential fibre functor.
\end{itemize}
\end{definition}

For a differential fibre functor $\omega : \Cat \to \Vect_{\K}$, we denote the differential tensor and tensor automorphisms of the functor by $$\Aut^{\Delta,\otimes}(\omega)\quad \text{and}\quad \Aut^{\otimes}(\omega),$$  respectively \cite[Definition 8]{OvchRecoverGroup} (generalized to several commuting derivations, \cite{GGO}) and \cite[pages 128--129]{Deligne}. It is shown in \cite[Theorem~2]{OvchTannakian} (together with isomorphisms~\eqref{eq:commutingderivations} that show that in the recovered differential structure the derivations in $\Delta$ commute) that every
neutral differential Tannakian category $(\Cat,\omega)$ is equivalent to the category of representations of a (pro-)LDAG $G$. In this case, viewed as representable functors,
$$
G \cong \Aut^{\Delta,\otimes}(\omega).
$$
Similarly to \cite[Proposition~2.20(b)]{Deligne}, it follows that $G$ is a LDAG if and only if the category $\Cat$ has 
one differential abelian tensor generator $X \in \Ob(\Cat)$. That is, $\Cat$ is the smallest subcategory in itself
containing $X$  and closed under $\At_{\Cat,i}$, $1 \Le i\Le m$, $\otimes$, $\oplus$, and taking duals and subquotients. In this case, similarly to the proof of  \cite[Proposition~2.20(b)]{Deligne}, under the above equivalence, $X$ is a faithful representation of the LDAG $G$ (see also \cite[Proposition~2]{OvchRecoverGroup}).  

\begin{theorem}\label{ZariskiClosure} Let a neutral
 differential Tannakian category
 $\Cat$
 have a differential rigid abelian tensor generator
 $X$. Then,
 there is a natural embedding of 
 $$G := \Aut^{\Delta,\otimes}(\omega)$$ into $$H_X := \Aut^\otimes(\omega|_\mathcal{D})$$
 so that $G$ is Zariski dense in $H_X$, where $\mathcal{D}$ is the rigid abelian
 tensor category generated by 
 the object $X$.
 \end{theorem}
 \begin{proof}
 Let $K$ be a $\Delta$-$\K$-algebra and $\lambda \in G(K)$. Since $\lambda$
 is uniquely determined by its action on $X$ (see \cite[formula (36)]{difftann}),
 the restriction map
 $$
 R : G \to \Aut(X)\otimes K,\quad \lambda \mapsto \lambda_X
 $$
 is injective. Similarly, $\lambda_X$ extends uniquely to a tensor automorphism of $\omega|_\mathcal{D}(K)$. This gives an embedding $$G \to H_X$$ functorial
 in $X$.
 By Tannaka's theorem (see \cite[Theorem~1]{Saavedra}, \cite[Theorem~2.11]{Deligne}, \cite[Theorem~2.5.3]{Springer}),
 \begin{align}\label{DRepH}
 \mathcal{D} = \Rep_{H_X}
 \end{align} and $H_X$ is a linear algebraic group with its faithful representation $\rho$ into $\GL(\omega(X))$.
Let $N_X$ be the Zariski closure of $G$ in $$\rho(H_X)\subset \GL(\omega(X)).$$ If
$$N_X \subsetneq \rho(H_X),$$ by Chevalley's theorem \cite[Theorem 5.1, Chapter II]{Borel}, there would be an $N_X$-invariant line $$L\subset\omega(Y), \quad Y\in\Ob(\mathcal{D}),$$  that was not $H_X$ invariant.
Hence, the line $L$ is also $G$-invariant and, therefore, corresponds to an object in $\mathcal{D}$. This  contradicts~\eqref{DRepH}.
 \end{proof}

Note that, starting with a different $X$, one can get different linear algebraic groups $H_X$ in which
$G$ is Zariski dense (see Example~\ref{ex:nonfaithful}). It is a question how to define a canonical $H_X$. If one asks
$X$ to be of the smallest possible dimension, are  different $H$'s coming out of this isomorphic? The answer is `Yes', and this will be resolved in Section~\ref{uniqueness}.

\subsection{Extending representations to Zariski closures}

Recall that a representation is called completely reducible if it is isomorphic to a direct sum of irreducible representations, that is, of representations with no non-zero proper subrepresentations.

\begin{theorem}\label{prolongation}
Let $G \subset \GL(W)$ be a LDAG. Then any completely reducible representation
$$\rho : G \to \GL(V)$$ extends to an algebraic representation $\bar\rho$ of $\overline{G}$. Moreover, if $\rho$ is faithful and $\overline G$ is reductive, then $\bar\rho$ is faithful.
\end{theorem}
\begin{proof}
By Theorem~\ref{ZariskiClosure}, the Zariski closure $\overline G$ is
given by 
the rigid abelian tensor category $\mathcal{D}_W$  generated by  $W$. Let
$$
\rho = \bigoplus_{j=1}^p\rho_j,\quad \rho_j : G\to \GL(V_j),\ 1\Le j \Le p,
$$
be the irreducible decomposition.
We will show that the G-module $V_j$ belongs to $\Ob(\mathcal{D}_W)$, $1\Le j\Le p$. This will imply the required result.

By \cite[Proposition 2]{OvchRecoverGroup}, the $G$-module $V_j$ is a subquotient
of several copies of
$$F^{i_1}_1(W)\otimes\ldots\otimes F^{i_{V,m}}_m(W)\otimes (W^*)^{\otimes j_V}.$$
Among all such presentations of $V_j$ choose one with the smallest maximal prolongation exponent. Denote this integer by $h$.
So, suppose we have $F_i^h(W)$ for some $i$, $1\Le i\Le m$, present
in a representation of $V_j$ of the smallest maximal order and $h \Ge 1$. We may also assume that
the degree of $F_i^h(W)$ with respect to $\otimes$ in this expression is the smallest possible. 

Then, $V_i$ can be viewed as the quotient $U/S$  for some $G$-modules
$$S \subset U \subset F_i^h(W)\otimes W' \oplus W'',$$ where $W''$ is free of $F_i^h(W)$. 
Recall the short exact sequence of $G$-modules:
\begin{equation*}
\begin{CD}
0@>>>F_i^{h-1}(W)@>\iota_{F_i^{h-1}(W),i}>>F_i^h(W)@>\pi_{F_i^{h-1}(W),i}>>F_i^{h-1}(W)@>>>0.
\end{CD}
\end{equation*}
Denote the morphisms of representations$$\left(\iota_{F_i^{h-1}(W),i}\otimes \id_{W'}\right)\oplus\id_{W''} : F_i^{h-1}(W)\otimes W'\oplus W''\to F_i^h(W)\otimes W'\oplus W''
$$ and
$$\left(\pi_{F_i^{h-1}(W),i}\otimes \id_{W'}\right)\oplus\id_{W''} : F_i^h(W)\otimes W'\oplus W''\to F_i^{h-1}(W)\otimes W'\oplus W''
$$
by $\iota$ and $\pi$, respectively. Also denote
$$
\iota\left(F_i^{h-1}(W)\otimes W'\oplus W''\right)
$$
by $W_{h-1}$. Since $U/S$ is irreducible, $\pi(U)/\pi(S)$ is either $\{0\}$ or isomorphic to $U/S$. The latter contradicts to the minimality of our choice of a ``prolongation-tensor'' presentation of $U/S$. Hence,
$$
\pi(U)=\pi(S).
$$
This implies that $$U/S = U\cap W_{h-1}/S\cap W_{h-1}.$$ Therefore, $$U/S\cong U'/S'\quad \text{with}\quad S'\subset U' \subset F_i^{h-1}(W)\otimes W' \oplus W''$$
contradicting the minimality again. Thus, $h=0$ and the $G$-module $V_j$ belongs to $\Ob(\mathcal{D}_W)$. Therefore, the representation $\rho$ extends to a representation $\overline\rho$ of $\overline G$.

Suppose now that $\overline G$ is a reductive linear algebraic group and the representation $\rho$
is faithful. By \cite[Proposition 2]{OvchRecoverGroup}, the $G$-module $W$
is an object in the differential Tannakian category generated by the $G$-module $V$. Since $\overline G$ is reductive, the $\overline G$-module $W$ is completely
reducible by \cite[Chapter~2]{SpringerInv}. Hence, $W$ is completely reducible
as a $G$-module, since every $G$-submodule of $W$ is a $\overline G$-submodule of $W$.

Therefore, by the first part of the proof, the faithful $\overline G$-module $W$ belongs to the Tannakian category generated by $V$. Hence, by \cite[Theorem~3.5]{Waterhouse}, the category $\Rep_{\overline G}$ is generated
as a Tannakian category by the $\overline G$-module $V$. Thus,
by \cite[Proposition 2.20]{Deligne}, the $\overline G$-module $V$ is faithful.
\end{proof}

\begin{example}\label{ex:nonfaithful} Consider the following faithful differential representation
of  $\Gm$ over $(\U,\partial)$, which is a reductive LDAG (see Section~\ref{reductivesection}):
$$
\rho :\ \Gm \to \GL_2(\U),\quad g \mapsto \begin{pmatrix}
g &\partial g\\
0 &g
\end{pmatrix},
\ g \in \U^*.
$$
One can show that the Zariski closure $$H_\rho := \overline{\rho(\Gm)} = \left\{\begin{pmatrix}a&b\\
0&a\end{pmatrix}\:\big|\: a\in \U^*,\: b\in\U\right\}\cong \Gm\times\Ga,$$
which is not a reductive linear algebraic group.
The representation
$$
\rho_1:\ \rho(\Gm) \to \GL_1(\U),\quad \begin{pmatrix}
g &\partial g\\
0 &g
\end{pmatrix}\mapsto \begin{pmatrix}g\end{pmatrix},\ g \in \U^*,
$$
is completely reducible and faithful. However, its extension to $H_\rho$  given by
$$
\overline{\rho_1}:\ H_\rho \to \GL_1(\U),\quad \begin{pmatrix}a&b\\
0&a\end{pmatrix} \mapsto \begin{pmatrix}
a
\end{pmatrix},\quad a \in \U^*,\ b \in \U
$$
is not faithful. Also note that $\overline{\rho_1(\rho(\Gm))} \cong \Gm$ and $H_\rho \cong \Gm\times \Ga$ are not isomorphic as linear algebraic groups. However, they are Zariski closures of the LDAG $\Gm$ in its two faithful differential representations.
\end{example}

\subsection{Unipotent linear differential algebraic groups}\label{unipotentsection}
We recall  now  a few basic facts about unipotent LDAGs.

\begin{lemma}\cite[Theorem 2]{CassidyUnipotent}\label{unipotentlemma} Let $G$ be a LDAG. The following statements are equivalent:
\begin{enumerate}
\item $G$ is conjugate to a differential algebraic subgroup of the special triangular group $\ST_n$.
\item\label{unipotentlemma2} $G$ contains no element of finite order greater than $1$.
\item\label{unipotentlemma3} $G$ has a descending normal sequence (each subgroup is normal in its
predecessor) of differential algebraic groups
$$
G = G_0\supset G_1\supset\ldots\supset G_N = \{e\}
$$
with $G_i/G_{i+1}$ isomorphic to a differential algebraic subgroup of $\Ga$.
\end{enumerate}
\end{lemma}

\begin{definition}\label{unipotentdef}
 Such $G$ is called a {\it unipotent} LDAG.
\end{definition}

\begin{remark}\label{unipotentremark}
By Lemma \ref{unipotentlemma}\eqref{unipotentlemma2}, the image  a  unipotent LDAG under an injective homomorphism is a unipotent LDAG. Therefore, the property of a LDAG $G$ being unipotent  does not depend on the embedding of $G$ into $\GL_n$.
\end{remark}

\begin{lemma}\cite[Theorem~4.3(b)]{Waterhouse}\label{Zariskinormal} Let $G$ be a differential algebraic subgroup of $\GL_n$ and let $H$
be a normal differential algebraic subgroup of $G$. Then, the Zariski closure $\overline H$
of $H$ is a normal algebraic subgroup of the Zariski closure $\overline G$ of $G$.
\end{lemma}

The proofs of the following two statements were provided to the authors by Phyllis Cassidy.

\begin{lemma}\label{lem:unipotent} The Zariski closure of a unipotent differential algebraic subgroup of
$\GL_n$ is a unipotent algebraic group.
\end{lemma}
\begin{proof} Let $G$ be a unipotent differential algebraic group. There exists $g \in \GL_n$ such that $$gGg^{-1} \subset \ST_n.$$ Let $\overline G$ be the Zariski closure of $G$. Since the conjugation by $g$ extends to $\overline G$, the linear algebraic group $g\overline Gg^{-1}$ is a subgroup of $\ST_n$. Thus, $\overline G$ is unipotent.
\end{proof}

 \begin{theorem} A LDAG $G$ contains a maximal normal unipotent differential algebraic subgroup $\Ru(G)$.
\end{theorem}
\begin{proof} Let $\Ru\left(\overline G\right)$ be the unipotent radical of $\overline G$. Recall that $\Ru\left(\overline G\right)$ is the maximal normal unipotent algebraic subgroup of $\overline G$.\footnote{Many authors add the requirement of connectedness to the definition of the unipotent radical. However, this holds automatically if the ground field has characteristic zero \cite[Corollary 8.5]{Waterhouse}.} Let $$K = \Ru\left(\overline G\right)\cap G.$$ Then, $K$ is a normal unipotent differential algebraic subgroup of $G$. Let $H$ be a normal unipotent differential
algebraic subgroup of $G$. By Lemma~\ref{Zariskinormal}, $\overline H$ is normal in $\overline G$. Since $\overline H$ is unipotent, $\overline H \subset \Ru\left(\overline G\right)$, and
$$
H \subset \Ru\left(\overline G\right)\cap G = K.
$$
Thus, $K$ is the maximal normal unipotent differential algebraic subgroup of $G$.
\end{proof}

\begin{definition}The subgroup $\Ru(G)$ is called the {\it unipotent radical} of $G$.
\end{definition}

\subsection{Reductive linear differential algebraic groups}\label{reductivesection}
We are now ready to introduce and study reductive LDAGs.

\begin{definition} A linear differential algebraic group $G$ is called {\it reductive} if its unipotent radical is trivial, that is, $\Ru(G) = \{e\}$.
\end{definition}

\begin{remark}
By Remark \ref{unipotentremark}, reductivity of a LDAG $G$ does not depend on its faithful representation.
\end{remark}

Recall that an additive category is called {\it semisimple} if, for every object $V$
and subobject $W\subset V$, there exists a subobject $U \subset V$ such
that $V = W\oplus U$. Since $\Char\K = 0$, the category of representations of a reductive algebraic group is semisimple \cite[Chapter~2]{SpringerInv}.

\begin{theorem}\label{semisimplethm} Let $G \subset \GL(V)$ be a LDAG over $\U$.
Then the category $\Rep_G$ is semisimple if and only if
$G$ is conjugate in $\GL(V)(\U)$ to a subgroup $$H\subset \GL(V)(C)$$ and $H$ is a
 reductive linear algebraic group, where $C$ is the field of constants of $\U$.
\end{theorem}
\begin{proof}
Let $G$ be a LDAG with $\Rep_G$ semisimple
and $V$ be its faithful representation. Since $V$ is a subrepresentation of $F_i(V)$, $1\Le i\Le m$,
the latter $G$-module is not irreducible. According to  \cite[Propositions~3]{OvchRecoverGroup}, if the $G$-module $F_i(V)$ is completely reducible, then $G$ is conjugate
to a group of matrices with constant entries with respect to $\partial_i$, $1\Le i \Le m$. Therefore, this direction follows from
the corresponding statement about linear algebraic groups \cite[Chapter 2]{SpringerInv}.

 Now, if $G$ is conjugate to the group $H$ of matrices with constant entries and  $H$ is reductive, the statement follows again from the representation theory
 of linear algebraic groups \cite[Chapter 2]{SpringerInv} and \cite[Propositions~2 and~3]{OvchRecoverGroup}.
  \end{proof}

\begin{definition}\cite[page 222]{CassidyClassification}
A connected differential algebraic group is called {\it simple} if it is non-commutative and has no nontrivial
connected normal differential algebraic subgroups.
\end{definition}
\begin{definition}\cite[page 222]{CassidyClassification}
A connected differential algebraic group is called {\it semisimple} if it has no nontrivial
connected normal commutative differential algebraic subgroups.
\end{definition}

The following example was suggested to the authors by Phyllis Cassidy.

\begin{example}\label{sl2example} Let $G = \SL_n$ over the ground field $(\U,\:')$. Consider the action of $G$ on its Lie algebra $\sl_n$
by the adjoint representation:
$$
A \mapsto gAg^{-1},
$$
where  $g \in G(\U)$ and $A \in \sl_n(\U)$. Let $H$ be the algebraic group $\SL_n\ltimes\sl_n$,
where by $\sl_n$ we mean the additive group of the Lie algebra, and the multiplication
is given by
$$
(g_1,X_1)\cdot(g_2,X_2) = \left(g_1g_2,X_1+g_1X_2g_1^{-1}\right).
$$ The subgroup $\SL_n\times \{0\}$ is a maximal semisimple algebraic subgroup, and the unipotent radical is $\{0\}\times\sl_n$. So, $H$ is not reductive. Let
$$G' = \left\{\left(g,g'g^{-1}\right)\big|\:g \in \SL_n(\U)\right\}.$$ $G'$ is a differential
algebraic subgroup of $H$ and is differentially rationally isomorphic to $\SL_n$ and
so is a non-commutative simple differential algebraic group, whose Zariski closure, $H$, is not even
reductive.
\end{example}

\begin{remark}
It is not surprising that it was enough
to prolong the usual representation of $\SL_n$ {\it once} to show
that the representations do not split: the group is conjugate
to constants if and only if the {\it first} prolongation splits.
\end{remark}

\section{Main result}\label{uniqueness}
In this section, we will show in Theorem~\ref{main} that the Zariski closure
of a reductive LDAG in its faithful representation of minimal dimension is
unique up to an isomorphism and is a reductive linear algebraic group. As Example~\ref{ex:nonfaithful} shows, the uniqueness fails without the hypothesis on minimal dimension.
\subsection{Preparation}

\begin{definition}
A group $\Gamma$ is a \emph{product} of its subgroups $M$ and $N$ if the product morphism $M\times N\to \Gamma$ is surjective and $mn=nm$ for all $m\in M$, $n\in N$. In this case, we write $$\Gamma=MN.$$ If $|M\cap N|<\infty$, we say that $\Gamma$ is an \emph{almost direct product} of $M$ and $N$, and we write $$\Gamma=M\cdot N.$$
\end{definition}

We will use the following result on the structure of reductive algebraic groups. For a LAG $G$, we denote  its commutator subgroup by $G'$. The connected component of $G$ is denoted by $G^\circ$. For a subgroup $H \subset G$, its centralizer in $G$ is denoted by $C_G(H)$. The center of $G$ is denoted by $Z(G)$  (for the definitions, see \cite[Chapters 6 and 10]{Waterhouse}). If $G$ is connected, so is $G'$ \cite[Theorem 10.1]{Waterhouse}. In what follows, we will be frequently referring to \cite{Borel}, where the notation is different. However, we choose the one that we have here as it is more commonly used. Our references to the results from~\cite{Borel} are very explicit to avoid possible confusion. 

\begin{proposition}\label{App2}
Let $H\subset\GL_n(\U)$ be a connected reductive algebraic subgroup and $N\subset H$ be a normal algebraic subgroup. Then $N$ is reductive and
\begin{enumerate}
 \item\label{App21} $H=H'\cdot T$, the commutator subgroup $H'$ is semisimple and $T=Z(H)^\circ$ is an algebraic torus.
 \item\label{App22} $H=NC_H(N)$.
 \item\label{App23} If $A\subset H$ is a normal algebraic subgroup and $B=C_H(A)$, then $$N=(N\cap A)(N\cap B).$$
\end{enumerate}
\end{proposition}

\begin{proof}
By \cite[Corollary IV.14.2(b)]{Borel}, $N^\circ$ is reductive, and, therefore, by \cite[Corollary 8.5]{Waterhouse}, so is $N$. Statement~\eqref{App21} now follows from \cite[Proposition IV.14.2(2,3)]{Borel} and \cite[Proposition IV.11.21]{Borel}.

By statement \eqref{App21} and \cite[Proposition IV.14.10(2b)]{Borel}, there is an isogeny
$$
\alpha\colon \tilde{H}=H_1\times\cdots\times H_k\times T\to H,
$$
where $H_i\subset H'$, $i=1,\ldots,k$, are connected simple normal subgroups of positive dimension.
Since $N\subset H$ is normal, so is the preimage $$\tilde{N}=\alpha^{-1}(N)\subset\tilde{H}.$$ Let $N_i$, $1\Le i\Le k$, be the image of $\tilde{N}$ under the projection $\alpha_i\colon\tilde{H}\to H_i$. We have $$\tilde{N}\subset\left(\prod N_i\right)\times T.$$ Since $N_i$ is normal in $H_i$, we have either $N_i=H_i$ or $N_i$ is finite. By \cite[Proposition IV.14.10(2c)]{Borel}, $\left(\tilde{N'}\right)^\circ\subset \tilde{H'}$ is a product of some $H_i$'s. Hence, if $N_i=H_i$, then $N_i\subset \left(\tilde{N'}\right)^\circ$. Denote the product of all finite $N_i$'s by $\Gamma$. Thus, we have 
\begin{equation}
\label{eq_normal0}
\tilde{N}\subset \left(\tilde{N'}\right)^\circ\times \Gamma\times T.
\end{equation}
By \cite[Lemma V.22.1(i,vi)]{Borel}, $\Gamma\subset Z(\tilde{H})$. Since $T\subset Z(\tilde{H})$, we have
\begin{equation}
\label{eq_normal1}
\tilde{N}\subset\left(\tilde{N'}\right)^\circ Z\left(\tilde{H}\right)\subset\tilde{N}^\circ Z\left(\tilde{H}\right).
\end{equation}
Since $\alpha$ is onto, $$\alpha\left(Z\left(\tilde{H}\right)\right)\subset Z(H).$$ Therefore, applying $\alpha$ to formula \eqref{eq_normal1}, we obtain
\begin{equation}
\label{eq_normal2}
N\subset N^\circ Z(H).
\end{equation}
Thus, $C_H(N^\circ)=C_H(N)$ and \cite[Proposition IV.14.10(1b)]{Borel} implies $$H=N^\circ C_{H}(N^\circ)=NC_H(N).$$ This proves statement~\eqref{App22}.

Let us prove statement~\eqref{App23}. By statement \eqref{App22}, we have $$H=AB.$$ It follows from the first inclusion in~\eqref{eq_normal0} (applied for $N=A$) that 
\begin{equation}
\label{eq_normal3}
A\subset \left(A'\right)^\circ Z(H).
\end{equation}
If $(A')^\circ\subset H'$ does not contain a simple normal subgroup $S\subset H'$, then $(A')^\circ$ and $S$ commute \cite[Proposition IV.14.10(2b,2c)]{Borel}.
Hence, if a simple normal subgroup $S\subset H$ does not belong to $A$, it belongs to  $C_H(A')^\circ$, which is equal to $C_H(A)$ (that is, $B$) by~\eqref{eq_normal3}.
Therefore, since $N_c:=(N')^\circ$ is a normal connected subgroup of  semisimple  $H'$, by \cite[Proposition IV.14.10(2c)]{Borel}, it is a product of simple normal subgroups of $H'$, and we have $$N_c=(N_c\cap A)(N_c\cap B).$$ It follows from the decomposition in statement~\eqref{App22} that the center of any normal subgroup of $H$ belongs to $Z(H)$. Then, by statement~\eqref{App21}, we have $$N^\circ=N_cZ\left(N^\circ\right)\subset N_cZ(H).$$ Now, \eqref{eq_normal2} implies that $$N\subset N_cZ(H).$$ Since $Z(H)\subset B$, $$N\subset N_cB$$ and
$$
N=N_c(N\cap B)=(N_c\cap A)(N\cap B)=(N\cap A)(N\cap B),
$$
which finishes the proof.
\end{proof}

For a LDAG $G\subset\GL_n(\U)$, we set $$G(C)= G\cap\GL_n(\C)$$
and call $G(C)$ a \emph{group of constants} of $G$. Isomorphic LDAGs do not necessarily have isomorphic groups of constants. (For instance, the homomorphism $\Ga(C)\to\Ga$, $x\mapsto ux$, where $u\in\Ga\setminus\Ga(C)$, has the image with the zero group of constants.) A \emph{Chevalley group} is a connected {\it noncommutative} semisimple linear algebraic group defined over $\QX$. 
We will use the following results of Cassidy restated for our situation.

\begin{theorem}\label{Cassidy} We have:
\begin{enumerate}
 \item\label{Cassidy1}\cite[Proposition 31]{Cassidy} Let $G$ be a dense differential algebraic subgroup of $\Gm^n$. Then $G$ contains $$(\Gm^n)(C).$$
 \item\label{Cassidy2}\cite[Theorem 19]{CassidyClassification} Let $G$ be a connected Zariski dense differential algebraic subgroup of a simple Chevalley group $S$. Then either $G=S$ or $G$ is conjugate in $S$ to $S(C')$, where $C \subset C' \subset \U$.
 \item\label{Cassidy3}\cite[Theorem 15]{CassidyClassification} Let $G$ be a dense differential algebraic subgroup of a  connected noncommutative semisimple LAG $H$. Then $G$ is an almost direct product of its connected simple normal subgroups $G_i$. The Zariski closure $$\overline{G_i}\subset H$$ is a connected simple normal algebraic subgroup.
\end{enumerate}
\end{theorem}

\begin{definition}\label{defproper}
A LDAG $G\subset\GL_n(\U)$ is called \emph{constant-dense}, if there is a LAG $G_0\subset\GL_n(\C)$ such that $$G_0\subset G\subset\overline{G_0}.$$
\end{definition}

Lemma~\ref{proper} will play a crucial role in the proof of Theorem \ref{main}, which is our main result.

\begin{lemma}\label{proper}
The isomorphism class of any connected reductive LDAG contains a constant-dense group.
\end{lemma}
\begin{proof} Let $G \subset  \GL_n(\U)$ be a connected reductive LDAG. Then the restriction of the homomorphism
$$
\overline{G} \to \overline{G}/\Ru(\overline{G}) =: H
$$
to $G$ is injective, because, by definition, $\Ru(\overline{G})\cap G = \{e\}$. Moreover, $H$ is a reductive linear algebraic group.
The above also follows from the first part of the proof of Theorem~\ref{main} without introducing a circular argument. Therefore, every connected reductive LDAG is isomorphic to a LDAG whose Zariski closure is a reductive linear algebraic group.

To prove the main statement, first consider the case
$$
H=H_1\times H_2\times\cdots\times H_r\times T,
$$
where $H_i$, $1\Le i\Le r$, are simple Chevalley groups and $T=\Gm^k$.
We will show that
$$
G=G_1\times G_2\times\cdots\times G_r\times Z,
$$
where $G_i\subset H_i$, $1\Le i\Le r$, $Z\subset T$, are dense subgroups.
Since the commutator subgroup $H'$ is closed in $H$, $$L=\overline{G'}\subset H'$$ and $L$ is a normal subgroup of $H$. The quotient map $$p\colon H\to H/L$$ is a homomorphism of algebraic groups. Hence, it takes dense subgroups to dense ones. Note that $p(G)$ is commutative, because $L\supset G'$. Then $H/L$ is so. Therefore, $L\supset H'$. Thus, $$L=H'.$$
Since $G'$ is semisimple, it follows from Theorem \ref{Cassidy}\eqref{Cassidy3} that $$G'=G_1\times G_2\times\cdots\times G_r,$$ where $\overline{G_i}=H_i$. By Theorem~\ref{Cassidy}\eqref{Cassidy2}, there is $h\in H'$ such that
$$
hG'h^{-1}\supset H'(C).
$$
Consider the projection $$q\colon H=H'\times T\to H'.$$ Since $G'$ is normal in $G$, $q(G')=G'$ is normal in $q(G)$. Then $$G'=q(G).$$ Indeed, $q(G)$ is a connected dense differential algebraic subgroup of $H'$. Therefore, $$q(G)=\prod_{i=1}^rQ_i,$$ where $Q_i$'s are the simple components. Then $G_i$ is a normal subgroup of $Q_i$. Since either of $G_i$ and $Q_i$ is conjugate in $H_i$ to $H_i(C')$ or $H_i$, where $C' \subset \U$, and $H_i(C')$ 
 is not normal in $H_i$ (for example, because $H_i$ is simple as a LDAG), we obtain $$G_i=Q_i.$$ Thus, $$q(G)=G'.$$ This implies that, for any $g \in G$, there exist $l \in G'$ and $z \in T$ such that
$$
g=lz.
$$
Hence, $$z=l^{-1}g \in Z=G\cap T.$$
We have $\overline{Z}=T$, because the projection $H'\times T\to T$ maps the dense subgroup $G$ to the subgroup $Z$, which has to be dense too.
By Theorem \ref{Cassidy}\eqref{Cassidy1}, $Z\supset T(C)$. Then
$$
hGh^{-1}\supset H(C).
$$

Consider now the case of a general $H$.
Let $\beta\colon H\to F$ be an isomorphism, where $F$ is given by polynomial equations with coefficients in $\QX$.
By \cite{Steinberg} and \cite[Corollary 1.3]{Demazure}, there is an isogeny
$$
\alpha\colon \tilde{H}=H_1\times H_2\times\cdots\times H_r\times T\to F
$$
defined over $\QX$ where $H_i$, $1\Le i\Le r$, are simple simply connected Chevalley groups.
Since $G$ is dense in $H$, the Kolchin connected component $\tilde{G}$ of $\alpha^{-1}(\beta(G))$ is a Zariski dense differential algebraic subgroup of $\tilde{H}$. By the above, there is $h\in \tilde{H}$ such that $$h\tilde{G}h^{-1}\supset \tilde{H}(C).$$ Since $\alpha$ is defined over $\QX$,
$$
\alpha(h)\beta(G)\alpha(h)^{-1}\supset F(C).
$$
Since $F$ is given by polynomial equations over $\Q$, $\Q \subset C$, and $C$ is algebraically closed, \cite[Corollary~AG.13.3]{Borel} implies that $F(C)$ is Zariski dense in $F$.
Thus, $G$ is constant-dense as it is isomorphic to $\alpha(h)\beta(G)\alpha(h)^{-1}$.
\end{proof}

\begin{lemma}\label{split} 
Let $G, G_0$ be as in Definition \ref{defproper} and $G$ be reductive. Let 
\begin{equation}\label{eq:exactseq}
\begin{CD}
0@>>>U@>\iota>>V@>\pi>> W@>>> 0.
\end{CD}
\end{equation}
be an exact sequence of finite-dimensional $G$-modules. If $$G_0=AB\quad\text{and}\quad\Hom_A(W,U) = \Hom_B(W,U) = 0,$$ then sequence~\eqref{eq:exactseq} splits.
\end{lemma}

\begin{proof}
We need to show that there is a $G$-equivariant homomorphism $p\colon W\to V$ such that $$\pi\circ p=\Id.$$ Note that such a homomorphism is $G$-equivariant if and only if $$p(W)\subset V$$ is a $G$-submodule. The group $G_0$ is reductive. Indeed, the Zariski closure of $\Ru(G_0)$ in $G$ belongs to $\Ru(G)$. Since $G$ is reductive, $G_0$ is so too. Then sequence~\eqref{eq:exactseq} splits as a sequence of $G_0$-modules. We will use the same notation $\iota,\pi,p$ for maps of $G_0$-modules.
Set $$A_1=\overline{A}\cap G\quad\text{and}\quad B_1=\overline{B}\cap G,$$ where $\overline{A}$ and $\overline{B}$ are the Zariski closures in $\GL_n(\U)$. Then, $$G=A_1B_1.$$ Indeed, in the Zariski topology induced on $G$ from $\GL_n(\U)$, $A_1,B_1\subset G$ are the closures of $A,B\subset G$, respectively. The product $A_1B_1$ is Zariski closed in $G$. Hence, it contains the closure of $AB=G_0$, which is $G$. Let $$p\colon W\to V$$ be a $G_0$-equivariant homomorphism such that $\pi\circ p=\Id$. We will show that $p(W)\subset V$ is preserved by $G$.

Let $b_1\in B_1$. The subspace $$(b_1\circ p)(W)\subset V$$ is $A$-invariant because $A$ and $B_1$ commute (by a property of the Zariski closures). The projection $$\Pi\colon V=U\oplus p(W)\to U$$ is a $G_0$-equivariant homomorphism. Hence, the map $$\Pi\circ b_1\circ p\colon W\to U$$ is an $A$-equivariant homomorphism. Since
$\Hom_{A}(W,U)=0$, $$\Pi\circ b_1\circ p=0,$$ which means that $b_1p(W)\subset p(W)$. Thus, $p(W)$ is stable under $B_1$. Similarly, it is preserved by $A_1$. Finally, since $G=A_1B_1$, $$Gp(W)\subset p(W),$$
which finishes the proof.
\end{proof}

\subsection{Main result} We are now ready to prove the main result.

\begin{theorem}\label{main}
Let $G\subset\GL_n(\U)$ be a reductive LDAG and $$\rho\colon G\to\GL(V)$$ its faithful representation of minimal dimension. Then:
\begin{enumerate}
 \item\label{main1} The representation $\rho$ is completely reducible.
 \item\label{main2} $H=\overline{\rho(G)}$ is a reductive LAG.
 \item\label{main3} The group $H$, up to an algebraic isomorphism, does not depend on $\rho$.
\end{enumerate}
\end{theorem}
\begin{proof}
Statements~\eqref{main2} and~\eqref{main3} of the theorem follow from statement~\eqref{main1}. Indeed, suppose that $\rho$ is completely reducible, that is, $$V=\bigoplus_{i=1}^kW_i,$$ where $W_i$, $1\Le i\Le k$, are simple submodules of the $G$-module $V$. We will show that $H$ is reductive. We have $$\rho(G)\subset H\subset\prod_i\GL(W_i).$$ The projection $$\pi_j\colon\prod_i\GL(W_i)\to\GL(W_j)$$ maps $\overline{\rho(G)}=H$ to $H_j=\overline{\pi_j(\rho(G))}$. Since $\pi_j$ is onto, it maps normal subgroups to normal subgroups. Therefore, $$\pi_j(\Ru(H))\subset\Ru(H_j).$$ We have $\Ru(H_j)=\{e\}$. Indeed, otherwise, by the Lie-Kolchin Theorem, the fixed point subspace $$W_j^{\Ru(H_j)}\subset W_j$$ has a positive dimension. Since $\Ru(H_j)$ is a normal subgroup of $H_j$, this subspace is invariant under $H_j$ and, therefore, under $G$. But this contradicts to the simplicity of the $G$-module $W_j$. Since $$\Ru(H)\subset\prod_i{\Ru(H_i)}=\{e\},$$ we have $$\Ru(H)=\{e\}.$$ Thus, $H$ is reductive.

Now, we deduce statement~\eqref{main3} from statements~\eqref{main1} and~\eqref{main2}. Let $\tau\colon G\to\GL(W)$ be another faithful representation of minimal dimension and $\Gamma=\overline{\tau(G)}$.
The differential homomorphism $$\tau\rho^{-1}\colon\rho(G)\to\tau(G)$$ determines a completely reducible faithful representation of $\rho(G)$ in $\GL(W)$, having minimal dimension. By Theorem~\ref{prolongation}, $\tau\rho^{-1}$ extends to an algebraic homomorphism $\alpha\colon H\to K$. Similarly, $$\rho\tau^{-1}\colon\tau(G)\to\rho(G)$$ extends to $\beta\colon K\to H$. The algebraic homomorphism $\alpha\beta$ is the identity on $\tau(G)$ and, therefore, it is trivial on $K=\overline\tau(G)$. Similarly, $\beta\alpha=\Id$. Hence, $H$ and $K$ are isomorphic algebraic groups.

We will prove statement~\eqref{main1} of the theorem. To do this, we reduce the general case to that of connected $G$. Suppose that the restriction of $\rho$ onto $G^\circ$ is completely reducible. Then the group $H_1=\overline{\rho\left(G^\circ\right)}$ is reductive. Moreover, since, by \cite[page 908]{Cassidy}, we have $\left | G/G^\circ \right |<\infty$, $$\left | H/H_1\right |<\infty.$$ This implies that $H$ is reductive and, thus, $V$ is a semisimple $H$-module. Since every simple $H$-submodule of $V$ is a simple $G$-submodule, the representation $\rho$ is completely reducible. 

Suppose that $G$ is connected. If $$\alpha\colon G_1\to G$$ is an isomorphism of LDAGs, then the complete reducibility of $\rho\alpha$ implies that of $\rho$. Therefore, by Lemma \ref{proper}, it suffices to prove Theorem \ref{main}\eqref{main1} for a constant-dense LDAG $G$. So, we suppose that $$G_0\subset G\subset\overline{G_0}=H,$$ where $G_0\subset\GL_n(\C)$ is a LAG. Fix a short exact sequence \eqref{eq:exactseq} of $G$-modules. We need to show that this sequence splits.

Note that $V$ is a faithful $G_0$-module of minimal dimension. Indeed, it is faithful, because it is faithful as a $G$-module. Suppose that $V_0$ is a faithful $G_0$-module of minimal dimension and show that $$\dim V_0=\dim V.$$ Since $G_0$ is a reductive LAG, it follows from Theorem~\ref{prolongation} that $V_0$ is a faithful $H$-module. Because of the minimality of the $G$-module $V$, $\dim V_0=\dim V$.

Let $A\subset G_0$ be the kernel of the action of $G$ on $U$ (such an action is further denoted by  $G_0\colon U$). By Proposition~\ref{App2}\eqref{App22}, we have $$G_0=AB,$$ where $B=C_{G_0}(A)$. By Proposition~\ref{App2}, the subgroups $A$ and $B$ are reductive. For a $\Gamma$-module $L$, where $\Gamma$ is a reductive algebraic group, we denote the submodule of invariants by $L^\Gamma$. There is a $\Gamma$-submodule $L_\Gamma\subset L$ such that $$L\simeq L^\Gamma\oplus L_\Gamma.$$ Therefore, $L_\Gamma$ is isomorphic to the direct sum of simple nontrivial $\Gamma$-modules. If $$\alpha\colon L\to M$$ is a homomorphism of $\Gamma$-modules, then $$\alpha\left(L^\Gamma\right)\subset M^\Gamma\quad\text{and}\quad\alpha\left(L_\Gamma\right)\subset M_\Gamma.$$ This follows from the fact that $\alpha(S)\simeq S$ or $\alpha(S)=0$ for any simple $\Gamma$-module $S\subset L$.

We will show that $$\Hom_A(W,U)=0.$$ Let $\alpha\in\Hom_A(W,U)$. We have $U=U^A$. Hence, $$\alpha(W)=\alpha\left(W^A\right)+\alpha(W_A)=\alpha\left(W^A\right),$$ because $\alpha(W_A)\subset U_A=0$. The $G_0$-module $$V_0=U\oplus W_A$$ is faithful (the $A$-module $W_A$ is $G_0$-invariant, because $A$ and $B$ commute). Indeed, $$\Ker(G_0\colon V_0)\subset\Ker(G_0\colon U)=A.$$ Since $V\simeq V_0\oplus W^A$ is faithful, $$\Ker(G_0\colon V_0)=\{e\}.$$ By the minimality of $V$, we obtain $V=V_0$, or equivalently, $W^A=\{0\}$. Therefore, $\alpha=0$.
Now we show that $$\Hom_B(W,U)=0.$$ Let $\alpha\in\Hom_B(W,U)$. Consider the $G_0$-module $$V_0=V/\Img\alpha,$$ and note that $V_0$ has a submodule isomorphic to $W$. We have $$\Ker(G_0\colon V_0)\subset\Ker(G_0\colon W)\stackrel{(*)}{=}\Ker(B\colon W)\subset\Ker(B\colon\Img\alpha).$$ The equality (*) follows from Proposition~\ref{App2}\eqref{App23}. Since $$V\simeq V_0\oplus\Img\alpha$$ and $V$ is faithful, $V_0$ is faithful. Therefore, $V_0=V$ and $\Img\alpha=0$. 
\end{proof}

\begin{corollary}\label{cor:main} A LDAG $G$ is reductive if and only if, for any differential rigid abelian tensor generator $X$ of $\Rep_G$ of minimal dimension, the rigid abelian tensor category generated by $X$ is semisimple.  
\end{corollary}

\begin{remark} One can show that a differential algebraic representation $\rho$ of $G=(\Gm)^k$ is in fact algebraic if the restriction of $\rho$ onto $G(C)$ has all irreducible subrepresentations pairwise non-isomorphic. One could expect this to be true for any reductive LDAG $G$. However, this is not the case. Indeed, consider the following faithful $4$-dimensional differential representation $\rho$ of $G=\SL_2$ given by
$$
\SL_2(\U)\ni\begin{pmatrix}
a & b\\
c& d
\end{pmatrix} \mapsto
\begin{pmatrix}
a^2 &2ab &b^2 & 2(ab'-a'b)\\
2ac&1&2bd&2(ad'-bc'-a'd+b'c)\\
c^2&2cd&d^2&2(cd'-c'd)\\
0&0&0&1
\end{pmatrix},
$$
which can be viewed as an action of $\SL_2$ on an invariant subspace of the space of
differential polynomials of degree $2$ and order up to $1$. The restriction of $\rho$ onto $G(C)$ is the sum of two non-isomorphic irreducible representations. But $\rho$ is not algebraic. This shows that the requirement of minimality is essential in Theorem~\ref{main}. In fact, this example has led to a new development of the differential representation theory of $\SL_2$ \cite{MinOvRepSL2}.
\end{remark}

\subsection{Characterization of simple LDAGs}
We will now provide a complete Tannakian characterization of simple LDAGs.
This description consists of two steps: characterize simple LDAGs in terms of
simple LAGs (Theorem~\ref{thm:diffsimple}) and then characterize simple LAGs themselves in terms of their
representations (Theorem~\ref{thm:algsimple}). The goal is to use this in developing algorithms computing Galois groups of linear differential equations with parameters.

For a $G$-module $V$, set
$$
T(V)=\bigoplus_{n=0}^\infty V^{\otimes n}.
$$
\begin{theorem}\label{thm:algsimple}
A connected noncommutative LAG $G$ is simple if and only if, for any non-trivial $G$-module $X$, there
exists a $G$-module $Y$  such that any $G$-module $Z$ is a subquotient in $T(X)\otimes Y$.
\end{theorem}
\begin{proof}
Let $G$ be a connected noncommutative simple LAG. Then the image of $G\to\GL(X)$ lies in $\SL(X)$. Set $$K = \ker(G \to \GL(X)).$$ According to \cite[Section 3.5]{Waterhouse}, any representation of $G/K$ is a subquotient of $T(X)$.

Let $A = \K[f_1,\ldots,f_n]$ be the Hopf algebra of $G$ and $A^K \subset A$ be the Hopf algebra of $G/K$. Then, 
\begin{equation}\label{eq:Agenerated}
A = A^K[f_1,\ldots,f_n].
\end{equation} 
Since $G$ is a simple linear algebraic group and $K$ is normal in $G$, $K$ is a finite group. Therefore, by \cite[Excercise~5.12]{Atiyah}, the ring $A$ is integral over the ring $A^K$. Hence, by~\eqref{eq:Agenerated} and \cite[Proposition~5.1]{Atiyah}, there exist $g_1,\ldots,g_m \in A$ such that $A$ is generated by $\{g_1,\ldots, g_m\}$ as an $A^K$-module. By \cite[Proposition~2.3.6]{Springer}, there exists a $\K$-finite dimensional $G$-submodule $Y$ of $A$ containing the set $\{g_1,\ldots,g_m\}$. 

Since $K\subset G$ is normal, $A^K$ is a submodule of the $G$-module $A$. The universal property of the tensor product yields the existence of a linear map
$$
\varphi\colon A^K\otimes Y\to A,\quad a\otimes y\mapsto ay,
$$ 
which is a map of $G$-modules, where $g(a\otimes y)=g(a)\otimes g(y)$ for all $g \in G$.
Let $Z$ be a finite-dimensional $G$-module. By \cite[Lemma 3.5]{Waterhouse}, $Z$ is isomorphic to a submodule of $A^r$ for some $r\Ge 0$. Consider the $G$-homomorphism
$$
\varphi^r\colon \left(A^K\otimes Y\right)^r\to A^r, \quad (x_1,\dots,x_r)\mapsto (\varphi(x_1),\dots,\varphi(x_r)),
$$
where $x_i\in A^K\otimes Y$, $1\Le i \Le r$.
We have $Z$ isomorphic to a quotient of $U=(\varphi^r)^{-1}(Z)$. Choose a basis $\{e_1,\ldots,e_p\}$ of $U$ over $\K$. There exist $b_{ij}^k \in A^K$ and $c_{ij}^k \in Y$, with $1\Le i\Le p$, $1\Le j \Le q$, and $1\Le k\Le r$, 
such that
$$e_i = \left(\sum_{j=1}^q b_{ij}^1\otimes c_{ij}^1,\dots,\sum_{j=1}^q b_{ij}^r\otimes c_{ij}^r\right),\ 1\Le i\Le p.$$
Let $V\subset A^K$ be a finite-dimensional $G$-submodule containing $a_{11}^1,\dots, a_{pq}^r$. We have $$U\subset (V\otimes Y)^n=V^n\otimes Y.$$ Since $K$ acts trivially on $V$, $V$ is a $G/K$-module and, therefore, is a quotient of a $G/K$-module $W\subset T(X)$. Thus, we have
$$
T(X)\otimes Y\supset W\otimes Y\to V^n\otimes Y\supset U\to Z,
$$
where all arrows correspond to surjective $G$-morphisms. Hence, $Z$ is a subquotient in $T(X)\otimes Y$. 

Conversely, suppose that $G$ is not simple. Let $K$ be a connected normal algebraic subgroup of $G$ having positive dimension. Set $X$ to be a nontrivial $G/K$-module and let $Y$ be a $G$-module as in the statement. Then $Y$ is $K$-faithful. (Otherwise, the action of $G$ on $T(X)\otimes Y$ would have a kernel implying the non-existence of faithful subquotients.) Set $$Z=Y\otimes Y.$$ By the assumption, there exists a $G$-module $V\subset T(X)\otimes Y$ and a surjective $G$-morphism $\varphi : V\to Z$.
  
Fix an algebraic subgroup $T\subset K$ with $\alpha: T \cong \Gm$ (called one-dimensional torus; such a subgroup exists, for example, \ by \cite[Corollaries~IV.11.5 and~IV.14.2]{Borel}). Then $Y$ is decomposed into the direct sum of $T$-submodules \cite[Chapter~III.8]{Borel}:
$$
Y_k=\left\{y\in Y\:|\: t(y)=\alpha(t)^k\cdot y,\ t \in T\right\}, \quad k\in\ZX. 
$$
Let $d$ be the highest weight of the action of $T$ on $Y$, that is, the largest non-negative integer such that $Y_d$ or $Y_{-d}$ has positive dimension. Since the action of $T$ on $T(X)$ is trivial, $d$ is the highest weight for the action of $T$ on $T(X)\otimes Y$. Since $Z$ is the sum of $Y_i\otimes Y_j$, the highest weight for $T$ on $Z$ is $2d$. Since $d>0$ (because $Y$ is $K$-faithful), this implies the non-existence of $V$ and $\varphi$. We arrive at a contradiction to the existence of $Y$. 
\end{proof}

Let $G$ be a  LDAG. For a $G$-module $V$, set
$$
T_D(V)=\bigoplus_{{n,p=0}\atop{1\Le i \Le m}}^\infty \left(F_i^pV\right)^{\otimes n}.
$$

\begin{theorem}\label{thm:diffsimple} Let $G$ be a LDAG. Let also $$\rho_1 : G\to \GL(V)\quad\text{and}\quad \rho_2 : G\to \GL(U)$$ be two faithful representations of $G$ with $\rho_1$ of minimal dimension and $\rho_2$ completely reducible. Then, the following statements are equivalent:
\begin{enumerate}
\item\label{diffsimple1} The LDAG $G$ is simple.
\item\label{diffsimple2} The linear algebraic group $\overline{\rho_1(G)}$ is simple.
\item\label{diffsimple3} The linear algebraic group $\overline{\rho_2(G)}$ is simple.
\end{enumerate}
If $G$ is connected noncommutative, then these are equivalent to
\begin{enumerate}
\setcounter{enumi}{3}
\item\label{diffsimple4} For any non-trivial $G$-module $X$, there
exists a $G$-module $Y$  such that any $G$-module $Z$ is a subquotient in $T_D(X)\otimes Y$.
\end{enumerate}
\end{theorem}
\begin{proof}
We will first show the equivalence of statements~\eqref{diffsimple1}--\eqref{diffsimple3}.
Let $G$ be simple. Since $\Ru(G)$ is a normal connected differential algebraic subgroup, $\Ru(G) = \{e\}$. Therefore, $G$ is a reductive LDAG. By \cite[page 230]{CassidyClassification}, there is a faithful representation 
$$
\rho_3: G \to GL(W)
$$
with $\overline{\rho_3(G)}$ being a simple linear algebraic group. Hence, $\rho_3$ is completely reducible. By the first part of the proof of Theorem~\ref{main}, we have
$$
\overline{\rho_1(G)} \cong \overline{\rho_2(G)} \cong \overline{\rho_3(G)}.
$$
Thus, the linear algebraic groups $\overline{\rho_1(G)}$ and $\overline{\rho_2(G)}$ are simple.

If $\overline{\rho_1(G)}$ is a simple linear algebraic group, then, by \cite[page~228, Corollary~2]{CassidyClassification}, the LDAG $G$ is simple.

Due to \cite[Proposition~2]{OvchRecoverGroup} and isomorphism~\eqref{eq:commutingreps}, the proof of the equivalence of statements~\eqref{diffsimple1} and~\eqref{diffsimple4} repeats the proof of Theorem~\ref{thm:algsimple} using the proof of Lemma~\ref{proper}, where it is shown that every infinite reductive (and, therefore, simple) LDAG contains a subgroup isomorphic to $\Gm(C)$.
\end{proof}

 \section{Differential Chevalley's theorem}\label{diffChevalley}
It is interesting that the non-differential Chevalley theorem (a linear algebraic group has a representation where it is defined as the stabilizer of a line) appears in a paper with Kolchin \cite{ChevalleyKolchin}, where they give two proofs: one uses differential Galois theory and the other is direct. We give a direct version of the proof of its differential analogue. This result can be also considered as a corollary of \cite[Proposition~14]{Cassidy} but we will provide our own full argument.

As above, let $(\K,\Delta)$ be a differential field not necessarily differentially closed with $\Char\K=0$ and let $\U$ be the differential
closure of $\K$.

\begin{theorem} Let $G \subset \GL(V)$ a LDAG defined over $\K$. Then
there exists a faithful differential representation $$\rho : \GL(V) \to \GL(E)$$ and a line $L \subset E$ defined over $\K$ such that
$$
\rho(G)(\U) = \{g \in \rho(\GL(V))(\U)\:|\: gL = L\}.
$$
\end{theorem}
\begin{proof} We shall closely follow \cite[Proposition~1.9 in Chapter~I and Theorem~5.1 in Chapter~II]{Borel} with
slightly different notation and differential modifications. 
Let $I = \{F\}$ be the
radical differential ideal in $$A :=\K\left\{y_{ij},1/\det\right\}$$ defining $G$ in $\GL(V)$.  By the Ritt-Raudenbush theorem \cite[Theorem 7.1]{Kap}, the set $F$  can be chosen to be finite.
Let $f \in F$ and $$\Delta_A : A \to A\otimes A$$ be the comultiplication inducing the
regular representation $$r : \GL(V) \to \GL(A).$$
Let $g, x \in \GL(V)(\U)$
and
$$
\Delta_A(f) = \sum_{i=1}^nf_i\otimes g_i
$$
with $f_i$ and $g_i \in A$, $1\Le i\Le n$.
We have
$$
\left(r_g(f)\right)(x)=f(x\cdot g) = \Delta_A(f)(x,g)= \sum_{i=1}^nf_i(x) g_i(g).
$$
Therefore, for every $g \in \GL(V)(\U)$, we have
$$
r_g(f) \in \Span_{\U}\{f_1,\ldots, f_n\} =: W_f.
$$
Let
\begin{align*}
\F_f = \{W\subset A\:|\:& W \text{ is a finite dimensional $\U$-vector space defined over $\K$} \\
&\text{and } r_g(f) \in W \text{ for all } g\in\GL(V)(\U) \}
\end{align*}
Since $W_f \in \F_f$ by definition, we have $\F_f \ne \varnothing$. Let $$N \in \F_f\quad \text{and}\quad h \in \GL(V)(\U).$$ Since $r_h$ is a $\K$-linear automorphism of $A$, the vector space $r_h(N)$
is defined over $\K$ and is finite dimensional over $\U$. Moreover, let  $$g \in \GL(V)(\U).$$ By definition, $$r_{h^{-1}g}(f)\in N.$$ Hence,  $$r_g(f)=r_{hh^{-1}g}(f)=r_h\left(r_{h^{-1}g}(f)\right)\in r_h(N).$$ Therefore,
\begin{equation}\label{rhNinF}
r_h(N) \in \F.
\end{equation} Let
$$
V_f = \bigcap_{N \in \F_f} N.
$$
Since each $N \in \F_f$ is a finite dimensional $\U$-vector space, $V_f$ is so as well.
Since $$f = r_e(f)\in N$$ for all $N \in \F_f$, we have $f \in V_f$. Inclusion~\eqref{rhNinF} implies that $$r_g(V_f)\subset V_f$$ for all $g \in \GL(V)(\U)$.
Each $N \in \F_f$ is defined over $\K$, that is, there exists a set $F_N$ of linear equations
with coefficients in $\K$ such that $N$ is the zero set of $F_N$ in $A$. The $\U$-vector space $V_f$ is the zero set of $$\bigcup_{N \in \F_f} F_N$$ in $A$ and, hence, is defined over $\K$.
Therefore, the finite set $F$ is contained in a finite
dimensional, defined over $\K$, $\GL(V)(\U)$-invariant $\U$-subspace $U$ (the smallest vector space containing all $V_f$, $f \in F$) under
the right translations.

 Since $\U\otimes_{\K} I$ and $U$ are $G(\U)$-invariant and $U$ is finite dimensional,
the vector space $$W := (\U\otimes_{\K} I)\cap U$$ is a finite dimensional $G(\U)$-invariant subspace of $\U\left\{y_{ij},1/\det\right\}$ defined over $\K$ that generates $\U\otimes_{\K} I$ as a radical differential ideal. Then,
$$
G(\U) = \{g \in \GL(V)(\U)\:|\: gW =W\}.
$$
Indeed, $$G(\U) \subset \{g \in \GL(V)(\U)\:|\: gW =W\}$$ by definition of $W$. Let $g \in \GL(V)(\U)$
be such that $gW = W$. Since $W$ generates $\U\otimes_{\K} I$ as a radical differential ideal, for
any $f \in I$ there exists $d \in \Z_{\Ge 0}$ such that
$$
f^d \in \left(\bigcup\limits_{p=1\atop{1\Le q \Le m}}^\infty F_q^p(W)\right)\K\{y_{ij},1/\det\}.
$$
Since $g$
is a {\it differential} automorphism of $\U\{y_{ij},1/\det\}$, we have
$$
gF_i^p(W) = F_i^p(W),\ 1\Le i\Le m,
$$
and, therefore,
$$(gf)^d \in \left(\bigcup\limits_{p=1\atop{1\Le q \Le m}}^\infty F_q^p(W)\right)\K\{y_{ij},1/\det\}.$$
Hence, $$g\U\otimes_{\K}I = \U\otimes_{\K}I$$ and $g \in G(\U)$.
Let $$q = \dim W,\quad E = \bigwedge^qU,\quad \text{and}\quad D = \bigwedge^qW \subset E.$$
If $E$ is not a faithful representation of $\GL(V)$, replace it with $E\oplus F$,
where $F$ is a faithful representation of $\GL(V)$. Now, as in \cite[Theorem 5.1, Chapter II]{Borel}, the result follows from a linear algebra statement \cite[Lemma 5.1, Chapter II]{Borel}.
\end{proof}

\section{Acknowledgments}
The authors are grateful to Phyllis Cassidy, Sergey Gorchinskiy, Charlotte Hardouin,  Michael Singer, and the referee for their very helpful suggestions.

\bibliographystyle{elsarticle-num}
\bibliography{reductive}
 \end{document}